\documentclass[reqno]{amsart}
\newcommand \datum {March 7, 2021, version 1.1}

\usepackage{amssymb,latexsym}
\usepackage{amsmath}
\usepackage{graphicx}
\usepackage[dvipsnames]{xcolor}
\usepackage{enumerate}
\numberwithin{equation}{section}

\theoremstyle{plain}
 \newtheorem*{namedtheorem}{\theoremname}
 \newcommand{\theoremname}{testing}
 \newenvironment{named}[1]{\renewcommand{\theoremname}{#1}
   \begin{namedtheorem}}
   {\end{namedtheorem}}
 \newtheorem{theorem}{Theorem}[section]
  
 \newtheorem{lemma}[theorem]{Lemma}

 \newtheorem{corollary}[theorem]{Corollary}
  
\theoremstyle{definition}
 \newtheorem{definition}[theorem]{Definition}
 
 \newtheorem{problem}[theorem]{Problem}
 
 \newtheorem{remark}[theorem]{Remark}

\newenvironment{enumeratei}{\begin{enumerate}[\upshape (i)]}%
                            {\end{enumerate}}
\newcommand \LineR [1] {\textup{Line}_{\tup{R}} #1} 
\newcommand \LineL [1]  {\textup{Line}_{\tup{L}} #1} 

\newcommand{\gls}{\overset{\tbf{.}} {+} }%glued sum

\newcommand \rlambda {\mathrel{\lambda}}
\newcommand \fullrect [1] {\textup{FulR}(#1)}
\newcommand \lftmin {^-\kern-1pt}

\newcommand \cardplus {\mathop{\dot\cup}}

\newcommand {\ordplus}{+}

\newcommand \Jir [1] {\textup{J}(#1)} 
\newcommand \Mir [1] {\textup{M}(#1)}

\renewcommand \phi{\varphi}
\newcommand \restrict [2] {{#1}\kern-1pt \rceil_{\kern-1pt #2}}

\newcommand \brho {{{\rho}}}

\newcommand \ideal[1]{\mathord\downarrow\, #1}
\newcommand \filter[1]{\mathord\uparrow\, #1}

\newcommand \bdia {\mathcal C_1}
\newcommand \cdia {\mathcal C_2}

\newcommand \inn {{\mathfrak n}}
\newcommand \inp {{\mathfrak p}}

\newcommand \foot [1] {\textup{Foot}(#1)}
\newcommand \Foot [1] {\textup{Foot}(#1)}

\newcommand \tuple [1] {(#1)}
\newcommand \pair [2] {\tuple{#1,#2}}
\newcommand \cornl [1] {\textup{lc}(#1) }
\newcommand \cornr [1] {\textup{rc}(#1) } 

\newcommand \Sn[1] {S_7^{(#1)}}

\newcommand \OLit [1] {\textup{OLit}(#1)}

\newcommand \Lit [1] {\textup{Lit}(#1)}

\newcommand \LLit [1] {\textup{LeftLit}(#1)}

\newcommand \RLit [1] {\textup{RightLit}(#1)}

\newcommand \LRoof [1] {\textup{LRoof}(#1)}
\newcommand \LFloor [1] {\textup{LFloor}(#1)}
\newcommand \RRoof [1] {\textup{RRoof}(#1)}
\newcommand \RFloor [1] {\textup{RFloor}(#1)}

\newcommand \Lamps[1] {\textup{Lamp}(#1)}
\newcommand \Lamp[1] {\textup{Lamp}(#1)}
 
\newcommand \rhgeomb {\brho_{\textup{Body}}}

\newcommand \rhfoot {\brho_{\textup{foot}}}
\newcommand \rhinfoot {\brho_{\textup{infoot}}}

\newcommand \rhalg {\brho_{\textup{alg}}}

\newcommand \body [1] {\textup{Body}(#1)}

\newcommand \jirdprec {\mathrel{\prec_{\kern-1pt\Jir D}}}

\newcommand \Max [1] {\textup{Max}(#1)}

\newcommand \Peak [1] {\textup{Peak}(#1)}

\newcommand\red[1]{{\textcolor{red}{#1}\color{black}}}

\newcommand \nothing [1] {}

\newcommand{\msf}{\mathsf}% math sans serif
\newcommand{\SM}[1]{\msf{M}_{#1}}
\DeclareMathOperator{\Con}{Con}
\newcommand{\SB}[1]{\msf{B}_{#1}}
\newcommand{\SC}[1]{\msf{C}_{#1}}
\newcommand{\lp}{\tup{(}}
\newcommand{\tup}{\textup}% text upright
\newcommand{\rp}{\tup{)}}
\newcommand{\jj}{\vee}% join
\newcommand{\mm}{\wedge}% meet
\newcommand{\tbf}{\textbf}% text bold
\newcommand{\setm}[2]{\{\,#1\mid#2\,\}}% set with a middle
\newcommand{\set}[1]{\{#1\}}% set 
\newcommand{\iso}{\cong}% isomorphic
\newcommand{\con}[1]{\tup{con}(#1)}
\newcommand{\es}{\varnothing}% the empty set
\newcommand\recyclebin[1]{}

\begin{document}

%\red{\hskip8cm\fbox{\datum}}

%\vskip0.3cm

\title[Congruence lattices of slim, planar, semimodular lattices]
{A new property of congruence lattices\\ of slim, planar, semimodular lattices}
\author[G.\ Cz\'edli]{G\'abor Cz\'edli}
\email{czedli@math.u-szeged.hu}
\urladdr{http://www.math.u-szeged.hu/~czedli/}
\address{University of Szeged, Bolyai Institute. 
Szeged, Aradi v\'ertan\'uk tere 1, HUNGARY 6720}

\author[G.\ Gr\"atzer]{George Gr\"atzer}
\email{gratzer@me.com}
\urladdr{http://server.maths.umanitoba.ca/homepages/gratzer/}
\address{University of Manitoba}

\begin{abstract} 
The systematic  study of   planar semimodular lattices started  in  2007
with a series of  papers by  G. Gr\"atzer and E. Knapp.  
These lattices have connections with  group theory and geometry. 
A planar semimodular  lattice  $L$  is  \emph{slim} if  $\SM{3}$ 
it is  not  a sublattice of  $L$. 
In  his 2016   monograph, ``The  Congruences of a  Finite  Lattice,  
A ~ \emph{Proof-by-Picture Approach}", 
the second author asked for a characterization of  congruence  lattices 
of slim, planar, semimodular lattices.
In addition to distributivity,  
both  authors have previously  found specific properties of  these   congruence lattices.  
In this paper, we present a  new  property, 
the  \emph{Three-pendant Three-crown Property}.
The proof is based on the first author's papers: 2014 (multifork extensions), 
2017 ($\bdia$-diagrams), and a recent paper (lamps),
introducing the tools we need.
\end{abstract}

\thanks{This research of the first author was supported by the National Research, Development and Innovation Fund of Hungary, under funding scheme K 134851.}

\subjclass {06C10\hfill{\red{\tbf{\datum}}}}

\dedicatory{Dedicated to B\'ela Cs\'ak\'any on his forthcoming ninetieth birthday}


\maketitle    

\section{Introduction}\label{S:Introduction}
\subsection{The Main Theorem}\label{MainTheorem} 

The book G. Gr\"atzer~\cite{CFL2} presents many results
characterizing congruence lattices of various classes of finite lattices, 
spanning 80 years, up to 2015. 
In particular, in 1996, G. Gr\"atzer, H. Lakser, and  E.\,T. Schmidt \cite{GLS98} 
started looking at the class of semimodular lattices
and were surprised: every finite distributive lattice can be represented 
as the congruence lattice of a \emph{planar} semimodular lattice.

The sublattice $\SM 3$ played a crucial role in the Gr\"atzer-Lakser-Schmidt construction,
so it was natural to ask (see Problems 1 in G. Gr\"atzer~\cite{gG16}, 
originally raised in G. Gr\"atzer~\cite{CFL2}) 
what happens if, in addition to planarity and semimodularity, 
we also assume that the lattice is \emph{slim}, that is,
it does not have $\SM 3$ sublattices.

\begin{problem}\label{P:main}%Problem~\ref{P:main}
What are the congruence lattices of slim, planar, semimodular lattices?
\end{problem}

We call a slim, planar, semimodular lattice an SPS lattice. 
A finite distributive lattice $D$ is \emph{representable by an SPS lattice $L$}
(in short, \emph{representable}) if~$D$ is isomorphic
to the congruence lattice $\Con L$  of an SPS lattice $L$.

We say that a finite distributive lattice $D$ satisfies 
the \emph{Three-pendant Three-crown Property}
if the ordered set $R_3$ of Figure~\ref{F:properties} has no cover-preserving
embedding into $\Jir D$.

Our paper continues the research in G. Cz\'edli~\cite{gCa} 
that presented four new properties of $\Con L$. We provide one more.

Now we can state our result.

\begin{named}{Main Theorem} 
Let $L$ be a slim, planar, semimodular lattice. 
Then $\Con L$ satisfies the Three-pendant Three-crown Property.
\end{named}

We have one more theorem in this paper.

\begin{theorem}\label{T:patch}
Let $n$ be a positive integer number and $L_1$, \dots, $L_n$ be 
slim, planar, semimodular lattices with at least three elements.
Then there exists a slim rectangular lattice $H$ 
and a slim patch lattice $L$ such that the following two isomorphisms hold:
\begin{align} \Con H&\cong \Con {L_1}\times \dots \times \Con{L_n},
\label{patch1}\\
\Con L&\cong \bigl(\Con {L_1}\times \dots \times \Con{L_n}\bigr) \ordplus \SB 2,\label{patch2}
 \end{align}
\end{theorem}

In \eqref{patch2}, the operation $+$ stands for ordinal sum $($the second summand is put on the top of the first one$)$.
We define rectangular lattices and patch lattices in Section ~\ref{S:notation}. 

\subsection{Background}\label{S:Background}%Section~\ref{S:Background}
G.~Gr\"atzer and E.~Knapp~\cite{GKn07}--\cite{GKn10} started the study of planar semimodular lattices. 
There are a number of  surveys of this field, 
see the book chapter G.~Cz\'edli and G.~Gr\"atzer~\cite{CG14} 
in G.~Gr\"atzer and F.~Wehrung, eds.\,~\cite{LTS1},
and G.~Cz\'edli and \'A.~Kurusa~\cite{CK19}.
For the topic: congruences of planar semimodular lattices, 
see the book chapter G.~Gr\"atzer~\cite{gG13b} 
in G.~Gr\"atzer and F.~Wehrung, eds.\,~\cite{LTS1}.

This research have also led to results outside of lattice theory:
to a group theoretical result by G.~Cz\'edli and E.\,T.~Schmidt~\cite{CS11} 
and G.~Gr\"atzer  and J.\,B.~Nation~\cite{GN10}, 
and to (combinatorial) geometric results by  G.~Cz\'edli~\cite{gC14b} and \cite{gC19}, 
K.~Adaricheva and G.~Cz\'edli ~\cite{AC14}, 
and  G.~Cz\'edli and \'A.~Kurusa~\cite{CK19}. 
G.~Cz\'edli and G.~Makay~\cite{CM17} presented a computer game based on these lattices. 
G.~Cz\'edli~\cite{gCb} is a related model theoretic paper.

The next two theorems summarize what we know about congruence lattices of SPS lattices. 
(In both theorems, the covering relations are those of the ordered set $\Jir{\Con L}$ and not of the lattice $\Con L$.)

\begin{theorem}[G.~Gr\"atzer ~\cite{gG16} and~\cite{gG20}]\label{T:Background}
Let $L$ be an SPS lattice with at least three elements.
\begin{enumeratei}
\item The ordered set $\Jir {\Con L}$ has at least two maximal elements.
\lp Equivalently, $\Con L$ has at least two coatoms.\rp\label{E:2M}
\item Every element of the ordered set $\Jir {\Con L}$ has at most two covers.
\end{enumeratei}
\end{theorem}

The three element chain is an example that the necessary condition
\eqref{E:2M} for representability is not sufficient.
G. Cz\'edli \cite{gC14a} gave an eight element distributive lattice
to~show that the necessary condition (ii) 
for representability is not sufficient.
See~also G. Gr\"atzer~\cite{gG15a}.

Our paper is a continuation of G.~Cz\'edli~\cite{gCa}. 
Here are some of the results of this paper.

\begin{theorem}[G.~Cz\'edli~\cite{gCa}]\label{T:Background1}
Let $L$ be an SPS lattice with at least three elements.

\begin{enumeratei}

\item {The set of maximal elements 
of the ordered set $\Jir {\Con L}$ 
can be represented as the disjoint union
of two nonempty subsets such that no two distinct elements 
in the same subset have a common lower cover.\label{E:LC} 

\item The ordered set $R$ of~ Figure~\ref{F:properties} 
cannot be embedded as a cover-preserving subset 
into the ordered set $\Jir {\Con L}$ provided that any maximal element of $R$
is a maximal element of $\Jir {\Con L}$.

\item If $x \in\Jir {\Con L})$ is covered by a maximal element $y$ of $\Jir {\Con L}$, 
then $y$ is not the only cover of $x$ in
the ordered set  $\Jir {\Con L}$.}

\item Let $x \neq y \in \Jir {\Con L}$, and let $z$ be a maximal element of $\Jir{\Con L}$. Assume that  both $x$ and $y$ are covered by $z$ in the ordered set $ \Jir {\Con L}$. Then there is no element $u \in \Jir {\Con L}$ such that $u$ is covered by $x$ and $y$ in  $\Jir {\Con L}$.
\end{enumeratei}
\end{theorem}
\begin{figure}[ht]
\centerline
{ \includegraphics[scale=1.0]{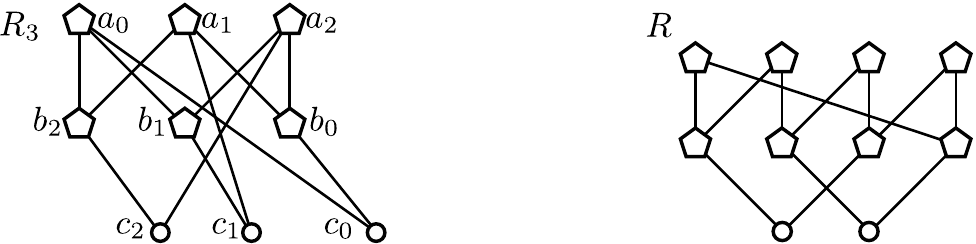}}
\caption{The Three-pendant Three-crown ordered set $R_3$
and the Two-pendant Four-crown ordered set $R$; 
the elements of the crowns are pentagons}\label{F:properties}
\end{figure}

\subsection*{Outline}Section~\ref{S:notation} recalls some concepts. 
Section~\ref{S:tools} recalls some of the tools developed in G.~Cz\'edli~\cite{gCa} while we develop some new tools in Section~\ref{sectionproofs}. 
We prove our Main Theorem in  Section~\ref{S: MainTheorem}. Finally, Section~\ref{S:Rectangular} proves  Theorem~\ref{T:patch} and discusses what we know about the congruence lattices of slim patch lattices.

\section{Basic notation and concepts}\label{S:notation}%Section ~\ref{S:notation}

All lattices in this paper are finite. 
We assume that the reader is familiar is with the rudiments of lattice theory. Most basic concepts and notation not defined in this paper are available in Part I of the monograph G. Gr\"atzer~\cite{CFL2}, which is free to access.
In particular, the \emph{glued sum} of two lattices $A$ and $B$ 
is denoted by $A \gls B$ ($B$ is on the top of $A$ with the unit element of $A$ and the zero of $B$  identified, so $\SC 2 \gls \SC 2$ is $\SC 3$).
The $n$-element chain is $\SC n$, the Boolean lattice with $n$ atoms is $\SB n$, 
and $\SM 3$ is the $5$-element modular nondistributive lattice.
The \emph{set of maximal elements} of an ordered set $P$ will be denoted by $\Max P$.
In this paper, edges are synonymous with prime intervals.
  
For a finite lattice $L$, the set of (non-zero) join-irreducible elements 
and (non-unit) meet-irreducible elements 
will be denoted by $\Jir L$ and $\Mir L$, respectively, so 
$\Jir L\cap \Mir L$ is the set of  \emph{doubly irreducible} elements.
We denote by $x^*$ the unique cover of $x$ for $x \in \Mir L$. 
For an element $a \in L$, let $\ideal a = \setm{x \in L}{x \leq a}$ 
be the \emph{principal ideal generated by $a$} 
and  $\filter a=\setm{x \in L}{x \geq a}$ the \emph{principal filter generated by $a$}. 

A planar semimodular lattice is \emph{slim} if it does not contain $\SM 3$ as a sublattice; 
see  G.~Gr\"atzer  and E.~Knapp~\cite{GKn07}, \cite{GKn09}, 
G.~Cz\'edli and E.\,T.~Schmidt~\cite{CS11}. 

Let $L$ be a planar lattice. A  \emph{left corner} $\cornl L$  
(resp., right corner $\cornr L$) of  $L$ is a doubly-irreducible element in $L - \set{0, 1}$ 
on the left (resp., right) boundary of~$L$. We~define a \emph{rectangular lattice} $L$ 
as a planar semimodular lattice which has exactly one left corner, 
$\cornl L$, and exactly one right corner, $\cornr L$, and they are complementary,
that is, $\cornl L \jj \cornr  L = 1$ and $\cornl L \mm \cornr L = 1$
(see  G.~Gr\"atzer  and E.~Knapp~\cite{GKn07}).
Finally, a rectangular lattice in which both corners are coatoms 
are called a \emph{patch lattice}.

\section{Tools}\label{S:tools}
In this section, $L$ is a slim rectangular lattice with a fixed $\bdia$-diagram, 
as we shall soon define.

We call the directions of $\pair 1 1$ and $\pair 1{-1}$ \emph{normal} and 
any direction $\pair{\cos\alpha}{\sin \alpha}$ with $\pi/2 < \alpha <3\pi/2$  \emph{steep}. (In \cite{gC1} and other papers, the first author uses ``\emph{precipitous}'' instead of ``steep''.)
Edges and lines parallel to a steep vector 
are also called steep, and similarly for normal slopes. 

The following definition and result are crucial in the study of SPS lattices.

\begin{definition}[G.~Cz\'edli \cite{gC1}]\label{D:well}
A diagram of the slim rectangular lattice $L$ is a
\emph{$\bdia$-diagram}
if it has the following two properties.
\begin{enumeratei}
\item If $x \in \Mir L - (\filter{\cornl L}\cup \filter{\cornr L})$, 
   then the edge $[x,x^*]$ is steep.
\item Every edge not of the form $[x,x^*]$ as in (i)
   has a normal slope.
\end{enumeratei}
If, in addition, 
\begin{enumeratei}\setcounter{enumi}{2}
\item any two edges on the lower boundary are of the same geometric length,
\end{enumeratei}
then the diagram is a \emph{$\cdia$-diagram}.
\end{definition}

\begin{theorem}[G. Cz\'edli \cite{gC1}]\label{T:well}
Every slim rectangular lattice $L$ has a $\cdia$-diagram.
\end{theorem}

The chains  $\ideal \cornl L$, $\filter \cornl L$,
$\ideal \cornr L$, and $\filter \cornr  L$ are called
the \emph{bottom left boundary chain}, 
\dots, \emph{top right boundary chain}. 
These chains have normal slopes 
and they are the sides of a geometric rectangle, 
which we call the \emph{full geometric rectangle} of $L$ and denote it by $\fullrect L$.
The four vertices of this rectangle are $0$, $1$, $\cornl L$, and $\cornr L$. 
The \emph{lower boundary} of $L$ is $\ideal \cornl L\cup  \ideal \cornr L$
and the \emph{upper boundary}  is $\filter \cornl L\cup  \filter \cornr L$.
With the exception of the corners,
no meet-irreducible element belongs to the lower boundary of $L$. 

The following is the central definition of G.~Cz\'edli~\cite{gCa}.

\begin{definition}\hfill

(A) Let $L$ be a slim rectangular lattice. 
The edges $[x, y]$ of $L$ with $x \in \Mir L$ are called \emph{neon tubes}. 
We call a neon tube $[x,y]$ on the upper boundary of $L$, 
a~\emph{boundary neon tube}; 
it is an \emph{internal neon tube}, otherwise. 
Equivalently, neon tubes with normal slopes are boundary neon tubes,
while steep neon tubes are internal. 

In Figures~\ref{F:lamps}, \ref{figbbsp}, and \ref{figgsch}, we represent the neon tubes by thick edges.

(B) A boundary neon tube $ \inn = [p, q]$ is also called a \emph{boundary lamp}.
This lamp~$I$ is an edge,  
the  neon tube $ \inn$ is the \emph{neon tube of the lamp~$I$}.  
Define $\foot I$ as $p$ and $\Peak I$ as $q$. 
If $\foot I$ is on the top left boundary chain, then $I$ is a \emph{left boundary lamp}; 
similarly, we define \emph{right boundary lamps}.  

In Figure~\ref{F:lamps}, the left boundary lamps and the right boundary lamps are $P_1,\dots, P_5$ and $Q_1,\dots,Q_6$, respectively, 
and $p_i=\foot{P_i}$ and $q_j=\foot {Q_j}$ for all $i$ and $j$.

(C) Every steep (that is, internal) neon tube $ \inn=[p,q]$ 
belongs to a unique \emph{internal lamp} $I=[\beta_q,q]$, 
where $\beta_q$ is the meet of all $p'\in L$ such that $[p' ,q]$ is a steep neon tube.  
For the lamp $I$, define the $\foot I$ as $\beta_q$ 
and the \emph{peak}  $\Peak I$ as~$q$.

In Figure~\ref{F:lamps}, there are five internal lamps, $A, \dots, E$ with 
$\foot A = a$, $\foot B = b $, and so on; also, $\Peak A=g$,  $\Peak B=h$, and  $\Peak C=z$; 
so $A = [a,g]$, $B=[b, h]$, and $C=[c,z]$. 

(D) The set $\Lamps L$ consists of all lamps of $L$. 
For example, for the lattice $L$ in Figure~\ref{F:lamps}, there are $16$ lamps in $L$.

(E) A lamp $I$  determines a geometric region 
(as in David Kelly and I.~Rival~\cite{KR75}) 
which we call the \emph{body} of $I$, and denote it by $\body I$. 
It has a geometric shape: it is either a line segment
or a quadrilateral whose lower sides have normal slopes
and whose upper sides are steep.
 
In Figure~\ref{F:lamps}, the regions $\body A$, $\body B$, and $\body C$
are colored dark-grey.
\end{definition}

\begin{figure}[ht]
\centerline
{ \includegraphics[scale=0.9]{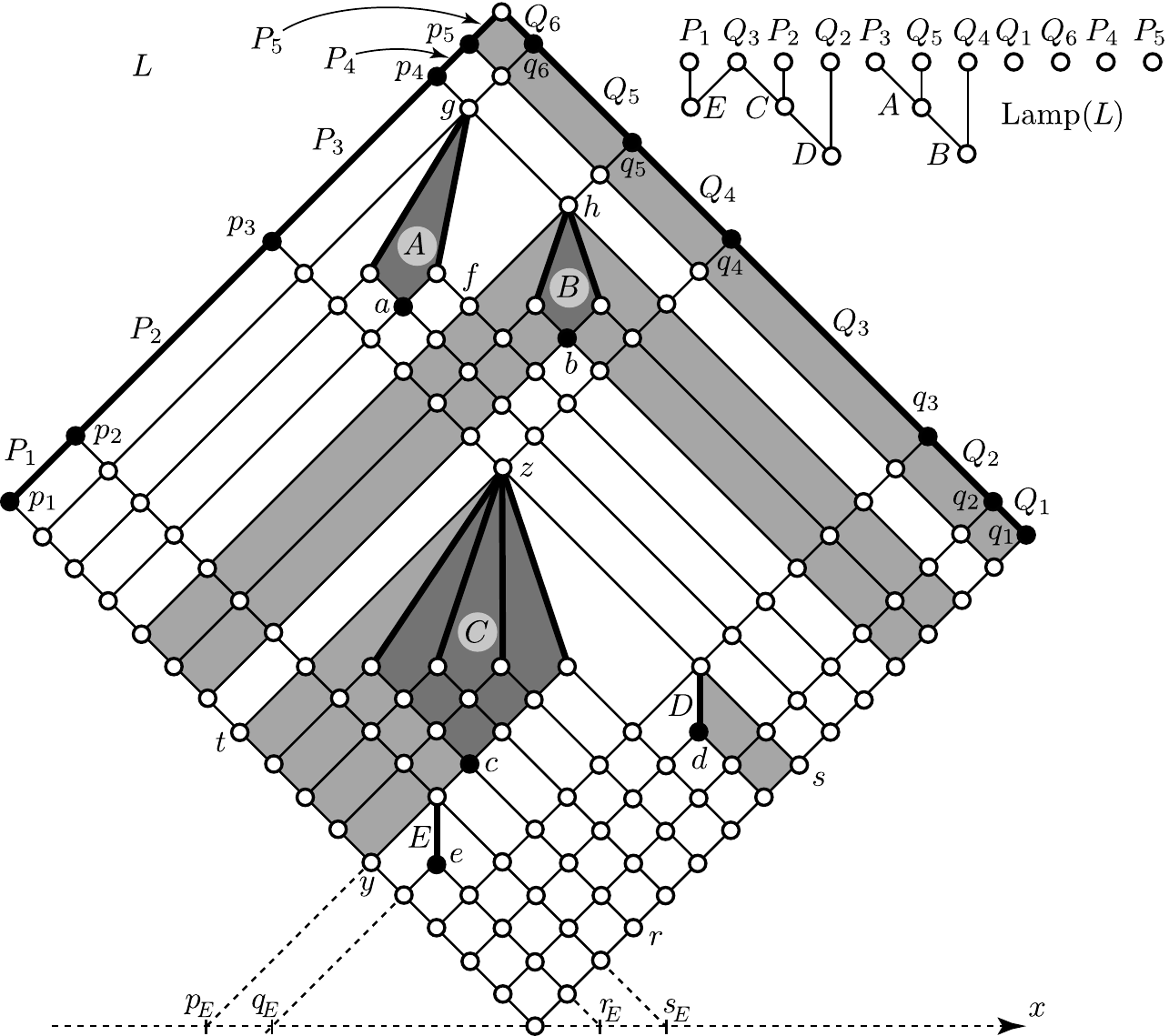}}
\caption{Lamps and related geometric objects}\label{F:lamps}
\end{figure}
For later reference, we recall by G.~Cz\'edli~\cite[Lemma 3.1]{gCa} that 
\begin{equation}
\text{A lamp is uniquely determined by its foot.}
\label{S1}
\end{equation}
The feet of our lamps are  black-filled in Figures~\ref{F:lamps}--\ref{figgsch}; 
this helps us find them.

In the real world, lamps emit light. Our lamps do it in a special way: 
the light rays go from all points of the neon tubes of a lamp~$I$  
downward with normal slopes. 
Next we give our definition of light emission. 
For an element $x \in L$, we define the line segment  $\LineL(x)$
from $x$ left and down, of normal slope 
to the lower-right boundary of $L$. 
Similarly, to the right, we have $\LineR(x)$.

So for a lamp $I$, we have the four line segments, 
from $\Peak I$ and $\Foot I$, left and right.
We denote them by $\LRoof I$ (the left roof), $\RRoof I$ (the right roof), 
$\LFloor I$, (the left floor)  and $\RFloor I$ (the right floor).

\begin{definition}[G.~Cz\'edli~\cite{gCa}]
For a lamp $I$ of a slim rectangular lattice $L$, we define
\begin{enumeratei}

\item \emph{the area left lit by $I$} (or, as in \cite{gCa}, \emph{illuminated from the right by $I$}), denoted by $\LLit I$, is a quadrangle bounded by
the line segments  $\LineL(\Peak{I})$, $\LineL(\Foot{I})$, 
the upper right boundary of $I$, and the appropriate line segment 
of the lower left boundary of $L$.
\label{D:lit1}
\item \emph{the area right lit by $I$}, denoted by $\RLit I$, is defined symmetrically.
\label{D:lit2}
\item \emph{the area lit by $I$}, denoted by $\Lit I$ is defined as $\LLit I\cup\RLit I$.
The geometric (topological) interior of $\Lit L$ is denoted by $\OLit L$ 
and we call it the \emph{open lit set} of $I$.
\label{D:lit3}
\end{enumeratei}
\end{definition}

For example, in Figure~\ref{F:lamps}, $\LLit C$, $\RLit D$, and $\Lit B$ are shaded.

It follows from the statements (2.10) and (2.11) of G.~Cz\'edli~\cite{gCa} that, 
for every lamp $I$ of $L$,
\begin{equation}
\parbox{8.6cm}{the geometric (that is, topological) boundaries of the\\ 
areas $\Lit I$, $\LLit I$, and $\RLit I$ consist of edges.}
\label{S:boundaries}
\end{equation}

Utilizing the concept of lit sets, we define some relations on $\Lamp L$; 
G.~Cz\'edli~\cite[Definition 2.9]{gCa} defines eight relations 
but here we only need four.

\begin{definition}[G.~Cz\'edli~\cite{gCa}]\label{deffRlsTs}
Let  $L$ be a slim rectangular lattice.
We define four relations $\rhgeomb$, $\rhfoot$, $\rhinfoot$, and $\rhalg$ 
on the set $\Lamps L$, 
by the following rules. For $I, J \in \Lamps L$,
\begin{enumeratei}
\item $(I ,J) \in \rhgeomb$ if  $I\neq J$, $\body I\subseteq \Lit J$,
and $I$ is an internal lamp;
\item $(I ,J) \in \rhfoot$ if $I \neq J$, $\Foot I \in \Lit J$, and $I$ is an internal lamp;
\item $(I ,J) \in \rhinfoot$ if  $I\neq J$, $\Foot I \in \OLit J$, and $I$ is an internal lamp;
\item $(I ,J) \in\rhalg$  if  $\Peak I\leq \Peak J$, $\Foot I \not\nleq\foot J$,
and $I$ is an internal lamp.
\end{enumeratei}
\end{definition}

The significance of lamps becomes clear from the following statement, which is a part of the (Main) Lemma 2.11 of G.~Cz\'edli~\cite{gCa}.

\begin{lemma}[G.~Cz\'edli~\cite{gCa}]\label{L:lamp}
Let $L$ be a slim rectangular lattice. Then $\rhgeomb = \rhfoot = \rhinfoot = \rhalg$.
Let $\brho$ stand for any one \lp or all\rp{} of these relations and 
let $\leq$ be the reflexive transitive closure of $\brho$.
Then $(\Lamp L, \leq)$ is an ordered set and it is isomorphic to $\Jir{\Con L}$.  
Also, if $I,J \in \Lamp L$ and $I\prec J$ in $(\Lamp L, \leq)$, 
then $\pair I J \in \brho$.
\end{lemma}

This lemma is illustrated by Figure~\ref{F:lamps}.
The isomorphism 
\[
   \Lamp L \iso \Jir{\Con L}
\]
is witnessed by the map
\[
\Lamp L \to \Jir{\Con L},\text{ defined by }
   I \mapsto \con {\Foot I,  \Peak I}. 
\]

We also need the following statement.

\begin{lemma}[G.~Gr\"atzer  and E.~Knapp~\cite{GKn09}]\label{L:recext}
If $K$ is a slim planar semimodular lattice with at least three elements, then there exists a slim rectangular lattice $L$ such that $\Con K \cong \Con L$.
\end{lemma}

G.~Gr\"atzer  and E.~Knapp~\cite{GKn09} proved a stronger statement, 
which we do not require. See also G. Gr\"atzer and E.\,T. Schmidt~\cite{GS14}.

To verify the Three-pendant Three-crown Property,
we have to work in $\Jir{\Con L}$.
So by utilizing Lemmas~\ref{L:lamp} and \ref{L:recext}, 
we can confine ourselves to investigate lamps in slim rectangular lattices. 

\section{Further tools and the Key Lemma}\label{sectionproofs}

\subsection{Coordinate quadruples}\label{S:quadruples}

We start with some technical tools.

\begin{definition}\label{D:lampnotation}
Let $I$ be a lamp of a slim rectangular lattice $L$ with a fixed $\bdia$-diagram. 
Assume that we choose the coordinate system of the plane $\mathbb R^2$  
so that $(0, 0)$ is the zero of $L$. 
\begin{enumeratei}
\item\label{D:lampnotationa}
Following G.~Cz\'edli~\cite{gCa}, the \emph{lit set} $\Lit I$ of an internal lamp $I$ 
is bordered by the line segments $\LRoof I$ and $\RRoof I$, 
$\LFloor I$, and $\RFloor I$,  
and the appropriate segments on the lower boundary. 
If $I$ is a boundary lamp, the above-mentioned line segments still border $\Lit I$. 
Any proper line segment lies on a line referred to as its \emph{carrier line}.
\item\label{D:lampnotationb} Let $(p_I, 0)$, $(q_I, 0)$, $(r_I, 0)$ 
and $(s_I, 0) \in\mathbb R^2$ be the intersection points of the $x$-axis 
with the carrier lines of $\LRoof I$, $\LFloor I$, $\RFloor I$, and $\RRoof I$, respectively. 
Then $\tuple{p_I, q_I, r_I, s_I}$ is called the \emph{coordinate quadruple} 
of the lamp~$I$. 
\item\label{D:lampnotationc} Let $I, J \in \Lamps L$.  
Then $I$ is \emph{to the left} of $J$, in notation $I \rlambda J$, 
if $q_I \leq p_J$ and $s_I \leq r_J$.
\end{enumeratei}
\end{definition}

For example, in Figure~\ref{F:lamps}, $\LRoof C$, $\RRoof C$,
$\LFloor C$, and $\RFloor C$ are the line segments corresponding to
the intervals (in fact, chains) $[t,z]$,  $[s,z]$,  $[y,c]$, and $[r,c]$,  respectively.
The coordinate quadruple of the lamp $E$ is shown in Figure~\ref{F:lamps} and, for example,  $E \rlambda D$ and  $P_1 \rlambda C$; however, $P_2 \rlambda C$ and $A \rlambda B$ fail. 
For $I \in\Lamps L$, the following observation follow from the definitions.
\begin{equation}
\parbox{9.2cm}{$p_I<q_I<r_I<s_I$ if and only if $I$ is an internal lamp,
\\$p_I=q_I<r_I<s_I$ 
if and only if $I$ is a left boundary lamp,
\\$p_I<q_I<r_I=s_I$ if and only if $I$ is a right boundary lamp.}
\label{E:3lamps}
\end{equation}

\begin{remark}\label{remjcRdnSt} Apart from an order isomorphism,
$(-p_I,s_I)$ and $(-q_I,r_I)$ are the join-coordinates of $\Peak I$ and $\foot I$ as in Cz\'edli~\cite[Definition 4.2]{gC1}.
\end{remark}

\subsection{Key Lemma}\label{S:Key}
The proof of the Main Theorem is based on the following key result.

\begin{lemma}[Key Lemma]\label{L:K} 
Let $I$ and $I'$ be  lamps of a slim rectangular lattice~$L$ 
with a fixed $\bdia$-diagram. 
If $I \neq I'$ and they have a common lower cover 
in $\tuple{\Lamps L;\leq}$, 
then either $I$ is to the left of $I'$ or $I'$ is to the left of $I$.
\end{lemma}

\begin{proof} For later use, recall the following statement G.~Cz\'edli~\cite[Corollary 6.1]{gC1}.
\begin{equation}
\parbox{8.7cm}{For $u \neq v \in L$, 
the inequality $u < v$ holds if and only if the ordinate 
(that is, the vertical $y$-coordinate) 
of $u$ is less than that of $v$ 
and the geometric line through $u$ and $v$ is 
either steep or it has a normal slope.}
\label{S:later} 
\end{equation}

In the rest of this proof, assume that $I \neq I'$ are lamps of $L$ 
and they have a common lower cover $I''$ 
and so incomparable, in notation, $I\parallel I'$.
By  Lemma~\ref{L:lamp}, both $\pair {I''}I$ and $\pair {I''}{I'}$ belong to $\rhinfoot$, 
that is, $\foot {I''} \in \OLit{I}$ and $\foot {I''} \in \OLit{I'}$. Hence, 
\begin{equation}
\OLit I\cap \OLit {I'}\neq\es.
\label{E:olit}
\end{equation} 
As Figure~\ref{F:lamps} (for the lamp $E$) shows or alternatively, as Remark~\ref{remjcRdnSt} yields, 
\begin{equation}
\parbox{8.3cm}{$(p_I, s_I)$, $(q_I, r_I)$, and $\tuple{p_I, q_I, r_I, s_I}$ 
determine $\Peak I$, $\foot I$, and $I$, respectively;}
\label{S:determine}
\end{equation}
and similarly for $I'$. Since $I$ and $I'$ are distinct, it follows from \eqref{E:olit} that
\begin{equation}
\parbox{9cm}{At least one of $I$ and $I'$ is not a left boundary lamp. 
Similarly, at least one is not a right boundary lamp.}
\label{E:notboundary}
\end{equation}

To make the proof more readable, we write $\tuple{p, q, r, s}$ for
$\tuple{p_I, q_I, r_I, s_I}$ and $\tuple{p', q', r', s'}$ for
$\tuple{p_{I'}, q_{I'}, r_{I'}, s_{I'}}$.

Statement \eqref{E:notboundary}
and G.~Cz\'edli~\cite[Lemma~3.8]{gCa} yield that 
\begin{equation}
q\neq q'\quad\text{ and }\quad r\neq r'.
\label{E:notequal}
\end{equation}

We distinguish several cases.

\emph{Case 1}: Both $I$ and $I'$ are internal lamps. 

We need the following concept (which is based
on the concept of circumscribed rectangles 
by G.~Cz\'edli~\cite[Definition 2.6]{gCa}) 
as visualized by Figure~\ref{F:lamps}. 
For an internal lamp $J \in \Lamps L$, the \emph{left shield} 
and the \emph{right shield} of $J$ are the left upper side 
and the right upper side of the circumscribed rectangle of $J$. 
So these shields are line segments. 
Namely, it follows from (2.8), (2.10), (2.14), 
and Definition 2.6 of G.~Cz\'edli~\cite{gCa} 
(and from the fact that $\foot J$ is in the interior of the  circumscribed rectangle of $J$) 
that
\begin{equation}
\parbox{8.3cm}{the right shield of an internal lamp $J$ is an edge of normal slope and this edge is longer than the geometric distance of (the carrier lines of) $\LRoof J$ and $\LFloor J$. Analogously for the left shield of $J$.}
\label{L:Key}
\end{equation}
Based on \eqref{L:Key}, there is another way 
to define the shields of an internal lamp $J$:
the \emph{left shield} of $J$ is the unique edge of slope $\tuple{-1,-1}$ 
whose top is $\Peak J$;
the \emph{right shield} of $J$ has slope $(1, -1)$ and its top is $\Peak J$. 
For example, in Figure~\ref{F:lamps}, $[h,g]$ is the right shield of $A$ 
while $[f,h]$ and $[y,\Peak E]$ are the left shields of $B$ and $E$, respectively.

\begin{figure}[ht]
\centerline
{\includegraphics[scale=1.0]{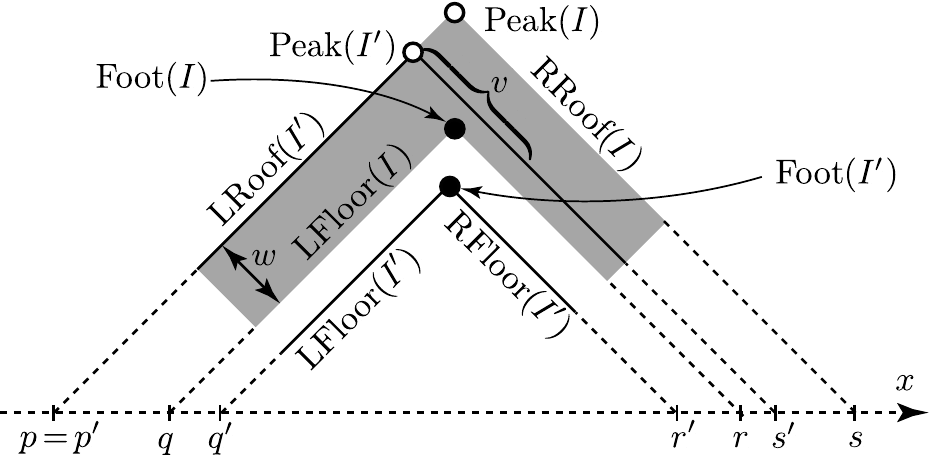}}
\caption{Proving \eqref{S:impl}}\label{figcsnkh}
\end{figure}

We know from (2.7) of G.~Cz\'edli~\cite{gCa} that distinct internal lamps have distinct peaks. This fact  along with \eqref{S1} and \eqref{S:determine} yield that 
\begin{equation}
\pair p s\neq \tuple{p',s'}\quad\text{ and }\quad \pair q r\neq\tuple{q', r'}.
\label{E:pairs}
\end{equation}
Next, we claim that 
\begin{equation}
p\neq p'.
\label{E:noteq}
\end{equation}

By way of contradiction, assume that $p = p'$. 
Since $q \neq q'$ by \eqref{E:notequal} 
and the role of $I$ and $I'$ is now  symmetric, 
we can assume that $q<q'$. 
Since  $p=p'$ and  \eqref{E:pairs} yield that $s\neq s'$, 
we conclude  that either $s<s'$ or $s'<s$. 

\emph{Case 1A}: $s'<s$. 
The situation (apart from the position of~ $r'$) 
is illustrated by Figure~\ref{figcsnkh}, where $\Lit I$ is the grey area and $\Lit{I'}$
is given by its boundary line segments $\LRoof{I'}$, $\LFloor{I'}$, etc.
The figure indicates the length
$v$ of the right shield of $I'$, 
which is greater than the ``width'' $(q'-p')/\sqrt 2$ of $\LLit {I'}$ by \eqref{L:Key}, 
and the ``width'' $w=(q-p)/\sqrt 2=(q-p')/\sqrt 2$ of $\LLit I$. 
By  \eqref{S:boundaries}, the geometric boundary of $\LLit I$ 
consists of edges (but these are not indicated in the figure 
between $\foot I$ and $\Peak I$). 
Since $v>w$, the geometric boundary of $\LLit I$ (consisting of edges) 
crosses the right shield of $I'$.
But this contradicts the planarity of the diagram 
since this right shield is an edge by \eqref{L:Key}.

\begin{figure}[ht]
\centerline
{\includegraphics[scale=1.0]{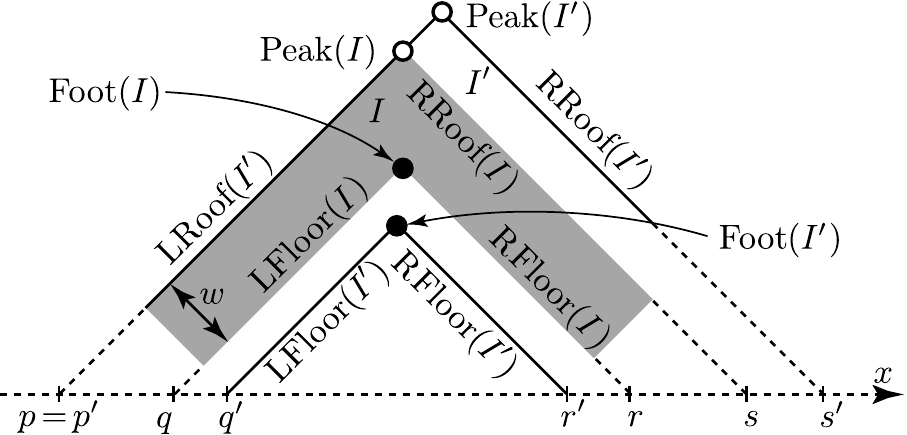}}
\caption{Still proving \eqref{S:impl}}\label{I:still}
\end{figure}
\emph{Case 1B}:  $s<s'$. This case is illustrated by Figure~\ref{I:still} 
(in which additional conditions hold, such as, $r' < r$).  
In this subcase, $q<q'$ yields that $\foot I$, 
which is on a line with point $\pair q 0$ and of slope $(1, 1)$ 
is above the carrier line of $\LFloor{I'}$.
Hence, it is clear  by the figure and, mainly by \eqref{S:later}, 
that $\Peak I\leq \Peak {I'}$ but $\foot I\nleq\foot{I'}$. Thus, $\tuple{I,I'} \in\rhalg$, contradicting that $I\parallel I'$. This 
completes \emph{Case 1B} and proves \eqref{E:noteq}.

\begin{figure}[ht]
\centerline
{\includegraphics[scale=1.0]{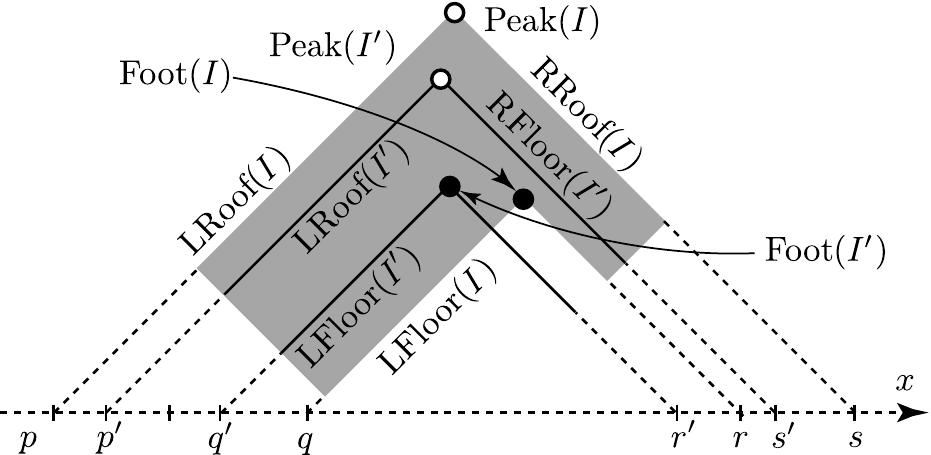}}
\caption{Case $p<p'<q$}\label{I:casep}
\end{figure}

Next, we are going to show that 
\begin{equation}
\text{if $p\leq p'$, then $q\leq p'$.}
\label{S:impl}
\end{equation}

So assume that $p\leq p'$. 
Then we know
from \eqref{E:noteq} that $p<p'$. 
By way of contradiction, assume that \eqref{S:impl} fails, that is, $p'<q$. 
Then neither $q'=q$, nor $q'>q$ by Lemma 3.8 of G.~Cz\'edli~\cite{gCa}. 
So $p<p'<q'<q$, see Figure~\ref{I:casep}. 
Observe that the geometric boundary of $\LLit {I'}$ 
cannot cross the right shield of~$I$ by \eqref{S:boundaries} and \eqref{L:Key}. 
So we obtain from \eqref{S:later} that $\Peak {I'}\leq \Peak I$. 
Note that $\foot {I'}$ is on the carrier line of $\LFloor {I'}$, 
which goes through the point $\pair{q'}{0}$; moreover, $q'<q$.
Therefore, \eqref{S:later} also yields that $\foot{I'}\nleq \foot I$. 
Hence, $\pair {I'}I \in\rhalg$, contradicting that $I\parallel I'$. 
This contradicts that $p'<q$ and so proves the validity of  \eqref{S:impl}.

As a variant of \eqref{S:impl}, observe that 
\begin{equation}
\text{if $s'\leq s$, then $s'\leq r$. Also, if $s\leq s'$, then $s\leq r'$.}
\label{S:also}
\end{equation}
Indeed, the first part of \eqref{S:also} follows from \eqref{S:impl} by left-right symmetry while its second part follows from the first part by interchanging the role of $I$ and $I'$.

\begin{figure}[ht]
\centerline
{\includegraphics[scale=1.0]{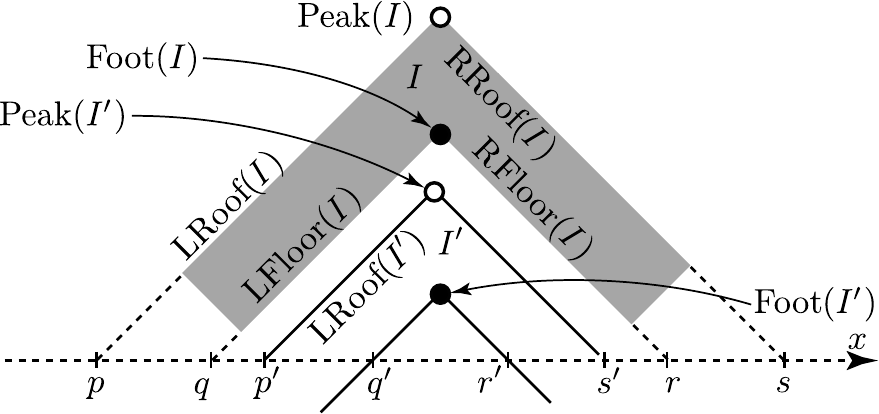}}
\caption{If $p<p'$ and $s'<s$}\label{pp}
\end{figure}

Next, we claim that  
\begin{equation}
\text{if }\, p\leq p',\,\text{ then }\, I\rlambda I'.
\label{E:pp'}
\end{equation}

So assume that $p\leq p'$. By  \eqref{S:impl}, we have that $q\leq p'$. 
We claim that $s\leq s'$. Assume to the contrary, that $s'<s$. 
By  \eqref{S:also}, $s'\leq r$; see Figure~\ref{pp}.  
Hence,  $\OLit I\cap \OLit {I'}=\es$, contradicting \eqref{E:olit}. 
This shows that $s\leq s'$. 
Applying \eqref{S:also}, we obtain that $s\leq r'$. 
Now, as part \eqref{D:lampnotationc} of  Definition~\ref{D:lampnotation} shows,  
$q\leq p'$ and  $s\leq r'$ complete the argument proving \eqref{E:pp'}.

Since  $I$ and $I'$ play a symmetric role,  we can assume that $p\leq p'$. Thus, \eqref{E:pp'} yields the validity of the lemma for  the case of internal lamps.

\begin{figure}[ht]
\centerline
{ \includegraphics[scale=1.0]{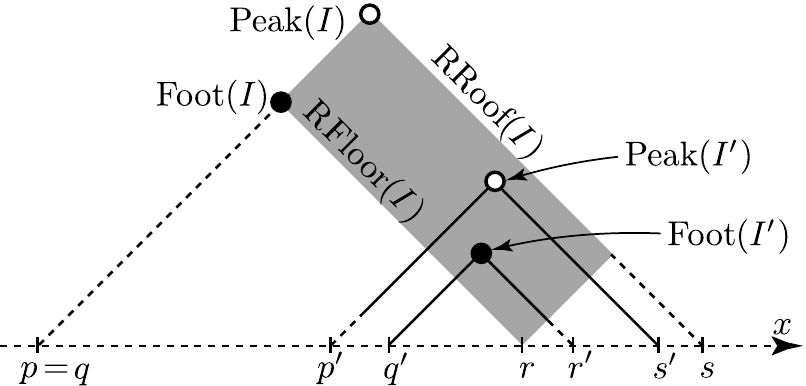}}
\caption{$I$ is a left boundary lamp and $r<r'<s$}\label{I:leftlamp}
\end{figure}

\emph{Case 2}: Of the two lamps, $I$ and $I'$, one is a boundary lamp 
and the other one is internal. 
By symmetry, we can assume that $I$ is a left boundary lamp and $I'$ is an internal lamp. By \eqref{E:3lamps}, $p=q$. Since this is clearly the least possible value, $q\leq p'$. Hence, to show that  $I\rlambda I'$, we need to show that $s\leq r'$. 

Suppose, for a contradiction, that $r'<s$. If $r' < r$, 
then we also have that $s'\leq r$; otherwise,
$\RFloor I$ would cross the left shield of $I'$ (see Figure~\ref{pp} 
after collapsing $p$ and $q$).
So if $r'< r$, then $s'\leq r$, but then  $\OLit I\cap \OLit {I'}=\es$ (similarly to Figure~\ref{pp}  but now $p=q$ and $\LLit I$ reduces to a line segment), and this equality contradicts   \eqref{E:olit}. This rules out that $r' < r$. Since $r'=r$ is also ruled out by
Lemma 3.8 of G.~Cz\'edli~\cite{gCa}, we have that  $r<r'$. 

So we have that $r<r'<s$; see Figure~\ref{I:leftlamp}.
Combining \eqref{S:later} and $r<r'$, we obtain that $\foot{I'}\nleq \foot I$. 
Thus $\Peak {I'}\nleq \Peak{I}$; 
indeed, otherwise we would have that $\pair {I'}I \in\rhalg$ 
and so $I'\leq I$ would contradict $I\parallel I'$. 
Since $s'\leq s$ (together with the trivial $p\leq p'$) 
would imply that $\Peak {I'}\leq \Peak{I}$, 
which has just been excluded, we obtain that, 
as opposed to what  Figure~\ref{I:leftlamp} shows, $s<s'$. 
However, then $r'<s<s'$ and $\RRoof I$ crosses the left shield of $I'$, 
which contradicts \eqref{S:boundaries}, \eqref{L:Key}, and the planarity of $L$. 
We have shown that $I\rlambda I'$, as required.

\emph{Case 3:} Both $I$ and $I'$ are boundary lamps. 
If they both were left boundary lamps, then $\OLit I\cap \OLit {I'}$ would 
contradict \eqref{E:olit}. 
We would have the same contradiction if both were right boundary lamps. Hence one of them, say $I$, is a left boundary lamp while the other, $I'$, is a right boundary lamp, 
and the required $I\rlambda I'$ trivially holds. 
This completes the proof of Lemma~\ref{L:K}. 
\end{proof}

\begin{figure}[ht]
\centerline
{ \includegraphics[scale=1.0]{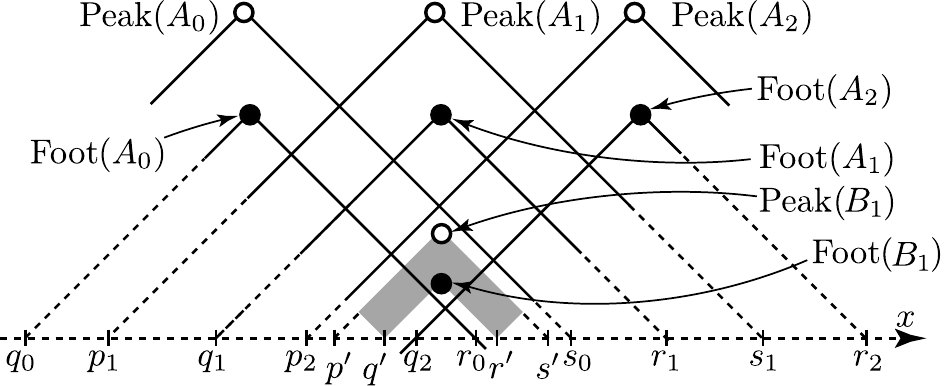}}
\caption{$A_0$, $A_1$, $A_2$, and $B_1$}\label{figkrtl}
\end{figure}

\section{Proving the Main Theorem}\label{S: MainTheorem}
%Section~\ref{S:MainTheorem}

Now we are ready to prove our main result.

\begin{proof}[Proof of the Main Theorem] The theorem is trivial for lattices with less than three elements. 
Hence, by Lemma~\ref{L:recext}, 
it suffices to prove the theorem for slim rectangular lattices. 
By way of contradiction, assume that $L$ is a slim rectangular lattice 
that fails the 3P3C-property. 
Then by Lemma~\ref{L:lamp}, $R_3$ is a cover-preserving ordered subset of $\Lamps L$. 
Let $X_i$ be the lamp corresponding to $x_i \in R_3$.
It follows from Lemma~\ref{L:K} that for any $i\neq j \in\set{0,1,2}$,
either $A_i$ is to the left of $A_j$ (in notation, $A_i\rlambda A_j$), 
or $A_j\rlambda A_i$. 
Therefore, since any permutation of $\set{A_0,A_1,A_2}$ 
extends to an automorphism of $R_3$, 
we can assume that $A_0\rlambda A_1$ and $A_1\rlambda A_2$; 
see Figure~\ref{figkrtl}, where the coordinate quadruple of $A_i$ 
is $\tuple{p_i,q_i,r_i,s_i}$. 
By Definition~\ref{D:lampnotation}\eqref{D:lampnotationc}, 
it follows that 
\begin{equation}
q_0\leq p_1,\,\,s_0\leq r_1,\,\,q_1\leq p_2,\,\,s_1\leq r_2,
\text{ and }p_i\leq q_i\leq r_i\leq s_i
\label{E:5ineq}
\end{equation}
for every $i \in\set{0,1,2}$. 
Note that $A_0$ is either an internal lamp such as in the figure, 
or it is a left boundary lamp and then $\LLit{A_0}$ is only a line segment, 
and analogously for $A_2$. 
Let $\tuple{p',q',r',s'}$ denote the coordinate tuple of $B_1$; 
note that $\Lit{B_1}$ is grey in the figure.  It follows from \eqref{E:5ineq} 
and from trivial properties of $\bdia$-diagrams that $\OLit{A_1}\cap\OLit{B_1} = \es$. 
On the other hand, $C_1\prec A_1$ and $C_1\prec B_1$ give that 
$\pair{C_1}{A_1} \in\rhinfoot$ and $\pair{C_1}{B_1} \in\rhinfoot$ by Lemma~\ref{L:lamp}.
It follows that $\foot{C_1} \in \OLit{A_1}\cap\OLit{B_1}=\es$, which is a contradiction, 
completing the proof of the Main Theorem.
\end{proof}

\section{Rectangular and patch lattices}\label{S:Rectangular}
%Section~\ref{S:xx}
Let  $(A,\rho)$ and $(A',\rho')$ be ordered sets. Their \emph{cardinal sum} will be denoted by 
$(A,\rho)\cardplus (A',\rho')$; 
it is $(A\sqcup A',\rho\sqcup\rho')$ where 
$\sqcup$
 stands for disjoint union.  
The operation $\gls$ for \emph{glued sum} was defined at the beginning of Section~\ref{S:notation}.

\subsection{Proving Theorem~\ref{T:patch}}\label{S:xx}
%Section~\ref{S:xx}

\begin{figure}[b!]
\centerline
{ \includegraphics[scale=0.9]{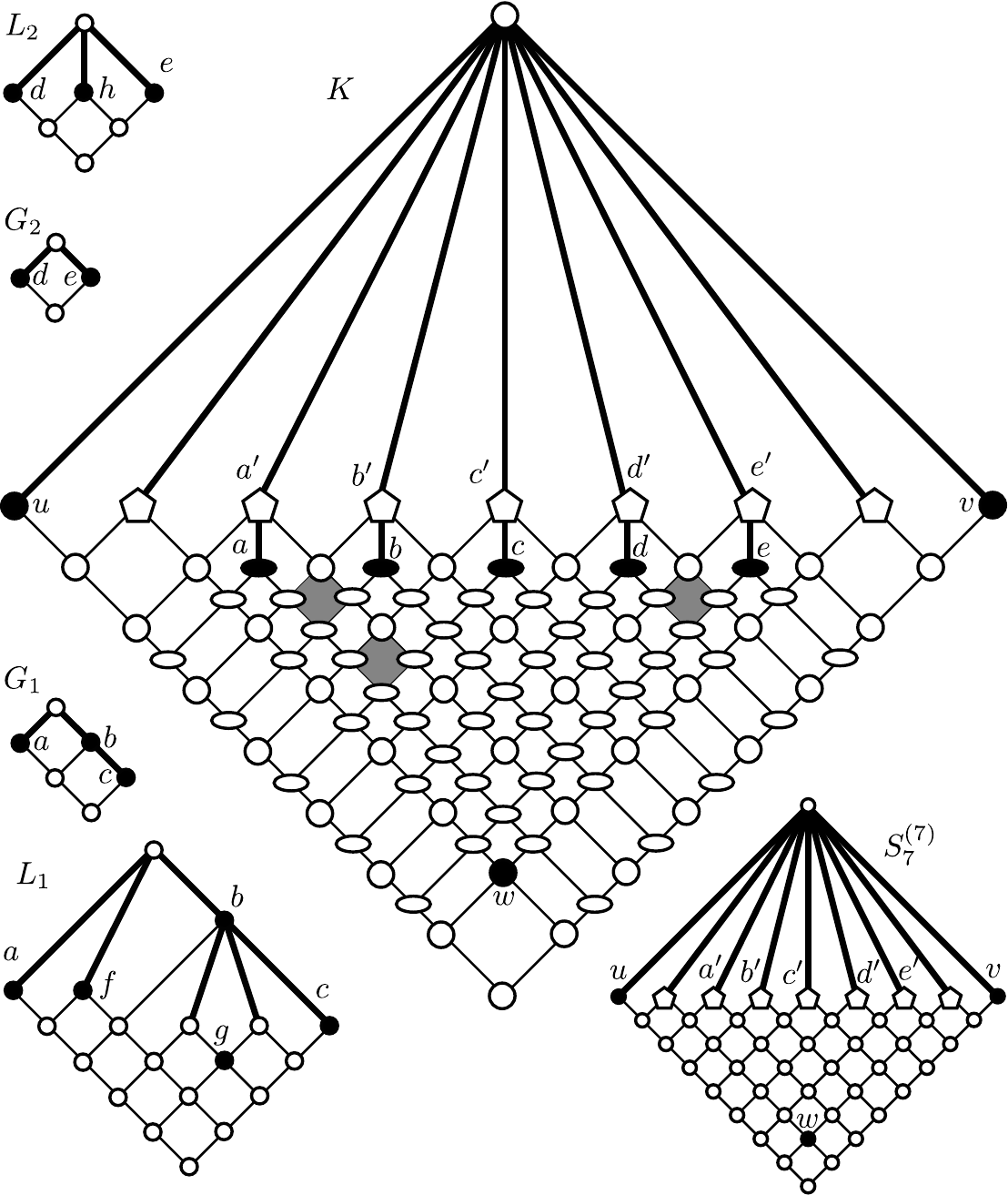}}
\caption{Constructing $K$ from $L_1$ and $L_2$}\label{figbbsp}
\end{figure}

By Lemma~\ref{L:recext}, we can assume that $L_1$, \dots, $L_n$ 
are slim rectangular lattices. 

First, we deal with \eqref{patch1}. It follows from G.\ Birkhoff's classical Representation Theorem of Finite Distributive Lattices, see, for example, G.~Gr\"atzer ~\cite[Theorem 107]{LTF}, 
that it suffices to find a slim planar semimodular lattice $H$ such that
\begin{equation}
\Jir H\cong \Jir{L_1}\cardplus\cdots\cardplus \Jir{L_n};
\label{disjunion}
\end{equation}
Let $H=L_1\gls \cdots\gls L_n$. 
Observe the each edge of $H$ is an edge of a unique  summand. 
For $i \in\set{1,,\dots,n}$, let $ \inp_i$ be an edge of $L_i$. 
This easily implies that, for $i,j \in\set{1,\dots,n}$, $\con{\inp_i}$ 
does not collapse $ \inp_j$ provided that $i\neq j$. 
The most convenient way to see this is by applying 
the Swing Lemma from G.~Gr\"atzer ~\cite{gG15}; 
see also  G.~Cz\'edli, G.~Gr\"atzer, and Lakser~\cite{CGL18} and 
G.~Cz\'edli and  G. Makay~\cite{CM17}. 
This implies \eqref{disjunion} and proves the \eqref{patch1}-part of the theorem.

\begin{figure}[t!]
\centerline
{ \includegraphics[scale=0.85]{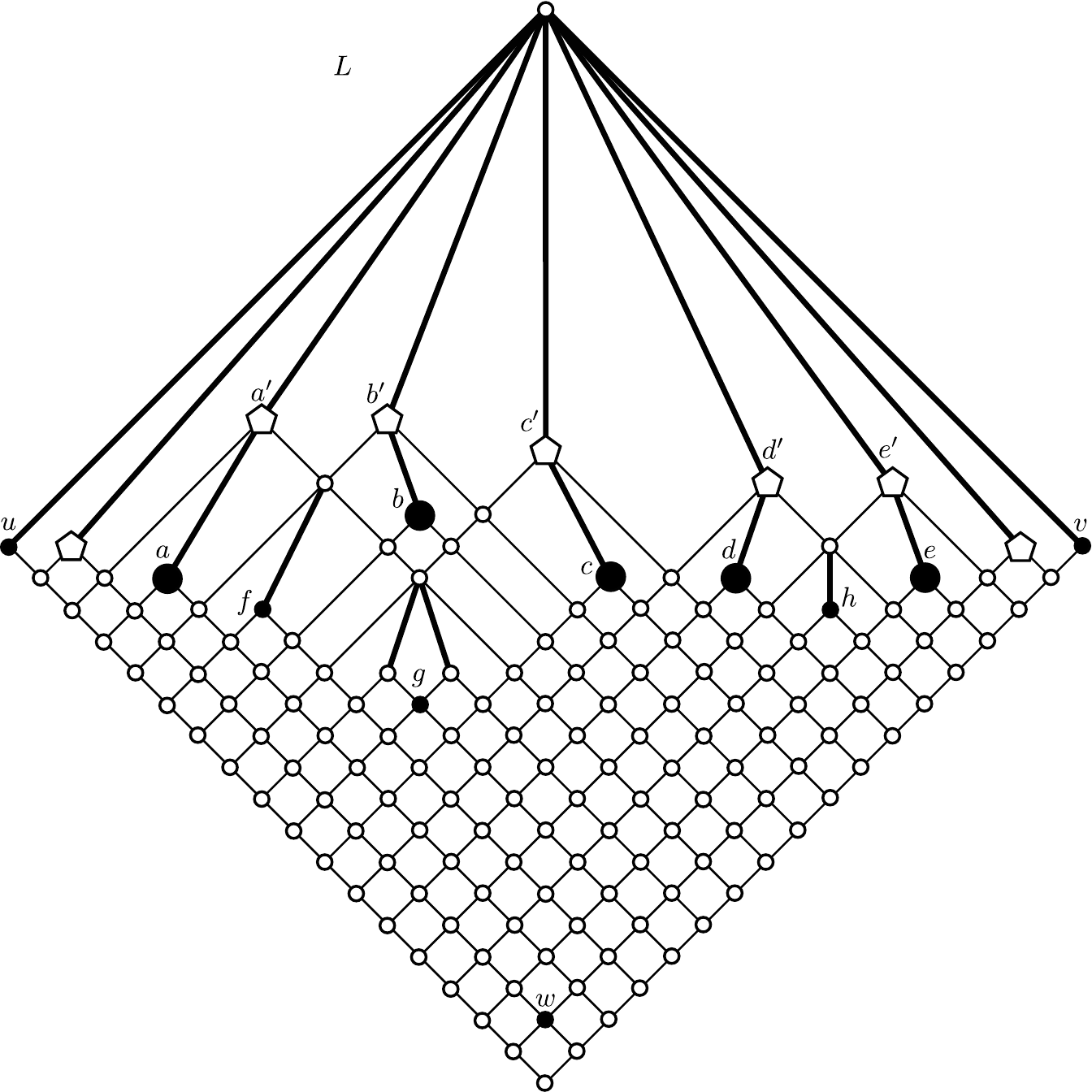}}
\caption{Constructing $L$ from $K$, $L_1$, and $L_2$}\label{figgsch}
\end{figure}

To prove \eqref{patch2}, we assume some familiarity with the multifork extensions of G.~Cz\'edli~\cite{gC14}. 
Recall that the \emph{grid} of the slim rectangular lattice $L_i$ 
is its sublattice generated by the upper boundary of $L_i$. 
This grid will be denoted by $G_i$; 
it is a distributive lattice with all if its edges of normal slopes.

For $i \in\set{1,\dots, n}$, let $t_i$ be the number of boundary lamps of $L_i$, and let $t=t_1+\cdots+t_n+2$. 
We start our  construction by taking $\Sn t$; see Figure~\ref{figbbsp}, 
where $n=2$, $t_1=3$, $t_2=2$, and $t=7$. 
The feet of the lamps are black-filled in the figure. 
Also, the feet of the internal neon tubes of $\Sn 7$ are pentagons.

Let $U$ and $V$ be the left boundary lamp and the right boundary lamp, respectively, 
of $\Sn 7$, and let $W$ be its unique internal lamp. 
In Figure~\ref{figbbsp}, the bottom of a lamp denoted by  a capital letter is denoted by the corresponding lower-case letter. 

Let $A, B, C, \dots$ be the list of $(t - 2)$ boundary lamps consisting, in this order, 
of the left boundary lamps of $L_1$, the right boundary lamps of $L_1$, the 
left boundary lamps of $L_2$, the right boundary lamps of $L_2$, \dots,
the left boundary lamps of~$L_n$, and the right boundary lamps of~$L_n$. 
Disregarding the leftmost one and the rightmost one, 
we label the feet of the neon tubes of $W$ by $a'$, $b'$, $c'$, \dots, 
from left to right, in this order.

By  G.~Cz\'edli~\cite[Proposition 3.3]{gC14},  $\Sn t$ is a slim patch lattice 
All the elements $a'$, $b'$, $c'$, \dots\ are the tops  of distributive 4-cells 
as defined in G.~Cz\'edli~\cite{gC14}. 
Insert a fork (that is, a 1-fold multifork) into each of these cells; 
the lattice we obtain is denoted by $K$; see Figure~\ref{figbbsp};
the elements of $K - \Sn 7$ (that is, the new elements) are oval. 
We know that $K$ is a slim patch lattice, see G.~Cz\'edli~\cite[Proposition 3.3]{gC14}.
The top edges of the forks just inserted are neon tubes and also 1-tube lamps; 
let $a,b,c,\dots$ denote their feet. 
For each $i \in\set{1,\dots,n}$, for each left boundary lamp $X$ of $L_i$ and for each right boundary lamp $Y$ of $L_i$, 
turn the intersection $\RLit X\cap \LLit Y$, which is a 4-cell, into 
grey; see Figure~\ref{figbbsp} again. Observe that,
\begin{equation}
\parbox{6.8cm}{disregarding the gaps among them, these grey 4-cells for a given $i$ are positioned in the same way as the 4-cells of the grid $G_i$.}
\label{S:gaps}
\end{equation} 

By G.~Cz\'edli~\cite[Theorem 3.7]{gC14},  
$L_i$ is obtained from $G_i$ by a sequence of multi\-fork extensions at distributive 4-cells
for $i \in\set{1,\dots,n}$.
 By \eqref{S:gaps}, the 4-cells of $G_1,\dots, G_n$ are in bijective correspondence with the grey 4-cells of $K$. This allows us to 
\begin{equation}
\parbox{7.8cm}{perform the multifork extensions in the same way in $K$   as in the procedure that turns $G_i$ to $L_i$}
\label{S:K}
\end{equation} 
for each $i \in\set{1,\dots, n}$. 
After performing the multifork extensions 
at some grey 4-cells of $K$ associated with $G_1,\dots, G_{i-1}$
(and possibly at some new 4-cells that earlier multifork extensions created), 
the grey 4-cells associated with~$G_i$ are still distributive. 
So now we begin with the grey 4-cells of $K$ instead of the 4-cells of~$G_i$. 

Let $L$ denote the lattice we obtain from $K$ with the multifork extensions 
as in \eqref{S:K}. As a continuation of  Figure~\ref{figbbsp}, $L$ is drawn in  Figure~\ref{figgsch}. The gaps mentioned in \eqref{S:gaps} cause no trouble since the light of neon tubes can go through them with no side effect. 
Since $\Lit W=\fullrect K=\fullrect L$ gives that $\pair A W,\pair BW, \pair C W,\dots$ belong to $\rhfoot$, we obtain (by Lemma~\ref{L:lamp}) that 
\begin{equation}
\text{the inequalities } A< W,\,\, B<W,\,\, C<W,\,\, \dots  \text{  hold in }\Lamps L.
\label{S:lampL}
\end{equation}
Let $H_i$ denote the set of lamps that are (in the geometric sense) 
in the grey 4-cells associated with $G_i$. 
Then $H_i$ is an ordered subset of $\Lamps L$. 
It follows easily from \eqref{S:K} that $H_i\cong \Lamps{L_i}$. 
Since light only goes in the directions $\tuple{-1,-1}$ and $\tuple{1,-1}$, 
we obtain that no lamp of $H_i$ lights up any  $\foot{H_j}$ for $i\neq j$. 
Thus we obtain that $\Lamps L - \set{U,V,W} = H_1 \cardplus \cdots \cardplus H_n$. 
This equality, $H_i\cong \Lamps{L_i}$, and \eqref{S:lampL} yield that
\begin{equation}
\Lamps L\cong \bigl(\Lamps{L_1}\cardplus\cdots\cardplus\Lamps{L_n}\bigr)\ordplus\set{U,V,W},
\label{S:UVW}
\end{equation}
where $W\prec U$, $W\prec V$, and $U\parallel V$. 
Finally, \eqref{S:UVW} and the Representation Theorem of Finite Distributive Lattices 
imply the validity of 
\eqref{patch2} and complete the proof of Theorem~\ref{T:patch}.

\subsection{Patch lattices}\label{S:Patch}
G.\ Gr\"atzer~\cite[Problem 3]{gG16} asks to characterize the congruence lattices of slim patch lattices. We now summarize what we know about these congruence lattices  but Problem 3 of G.\ Gr\"atzer~\cite{gG16} remains open.
We start with an observation.

\begin{lemma}\label{L:Keylemma}
If $L$ is a slim rectangular lattice, then the following three conditions are equivalent.
\begin{enumeratei}
\item $L$ is a slim patch lattice.
\item $\Jir{\Con L}$ has exactly two maximal elements. 
\item There is a finite distributive lattice $D_0$ such that $\Con L\cong D_0 \gls \SB 2$.
\end{enumeratei}
\end{lemma}

\begin{proof} By G.~Cz\'edli~\cite[Lemma  3.2]{gCa}, the maximal elements of $\Lamps L$ are exactly the boundary lamps. Hence, Lemma~\ref{L:lamp} implies that (i) is equivalent to (ii). 
This equivalence also easily follows from the Swing Lemma, see G. Gr\"atzer~\cite{gG15}.
Also,  the fact that  (ii) equivalent to (iii) holds by 
the Structure Theorem of Finite Distributive Lattices.
\end{proof}

The \eqref{patch2}-part of Theorem~\ref{T:patch} establishes a new connection between slim rectangular lattices and slim patch lattices; other connections have been explored by G.~Cz\'edli~\cite{gC14} and G.~Cz\'edli and E.\,T.~Schmidt~\cite{CS13}.

The four element boolean lattice $\SB 2$ and  the glued sum construction in part (iii) of Lemma~\ref{L:Keylemma} are well understood. 
So we focus on $D_0$ to describe the known properties of congruence lattices of slim patch lattices. The next statement reduces seven \emph{known} conditions that hold for congruence lattices of slim, planar, semimodular lattices by Theorems~\ref{T:Background}--\ref{T:Background1} and the Main Theorem  to four. 

\begin{corollary}\label{C:patch}%Corollary~\ref{C:patch}
Let $D=\Con L$ be the congruence lattice of a slim patch lattice $L$. Then the following four statements hold.
\begin{enumeratei}
\item\label{pRtswWa} There exists a unique finite distributive lattice $D_0$ such that $D=D_0\gls \SB 2$.
\end{enumeratei}
In the next three statements, $D_0$ refers to the distributive lattice defined in \eqref{pRtswWa}.
\begin{enumeratei}\setcounter{enumi}{1}
\item\label{pRtswWb}  Every element of the ordered set $\Jir {D_0}$ has at most two covers.
\item\label{pRtswWc}  Two distinct maximal elements of the ordered set $\Jir {D_0}$ have no common lower cover.
\item\label{pRtswWd}  The ordered set $\Jir {D_0}$ satisfies the Three-pendant Three-crown Property.
\end{enumeratei}
Furthermore, if $L$ is a finite lattice, $D=\Con L$, and $D$ satisfies  \eqref{pRtswWa}, \eqref{pRtswWb} and \eqref{pRtswWc}  above, then $L$ also satisfies all the six properties listed in Theorems~\ref{T:Background} and \ref{T:Background1}.
\end{corollary}

\begin{proof} 
Part \eqref{pRtswWa} follows from Lemma~\ref{L:Keylemma}. Let $\msf A_2$ denote the two element antichain. Observe that if $D=D_0\gls \SB 2$, then $\Jir D=\Jir{D_0}\gls \msf A_2$. Hence,
applying
Theorem~\ref{T:Background}(ii), 
Theorem~\ref{T:Background1}(iv), and the Main Theorem, we obtain parts \eqref{pRtswWb}, \eqref{pRtswWc}, and \eqref{pRtswWd}, respectively. The rest of the corollary is a trivial consequence of $\Jir D=\Jir{D_0}\gls \msf A_2$.
\end{proof}

\end{document}